\documentclass{article}

\usepackage{geometry}
\usepackage{paralist,graphics,epsfig,graphicx,epstopdf,mathrsfs}
\usepackage{float,color,ulem,comment,tabulary,booktabs}

\newcommand\doubleRule{\toprule\toprule}

\usepackage{fourier}
\usepackage{geometry}
\usepackage{amssymb}
\usepackage{amsfonts}
\usepackage{graphicx}
\usepackage{amsmath}
\usepackage{bbm}
\usepackage{stackengine}
\usepackage{epsf,epsfig,subfigure}
\usepackage{float,color,ulem}
\usepackage[colorlinks=true]{hyperref}
\usepackage{graphicx,multirow}
\usepackage{epsf,epsfig,subfigure,diagbox}
\usepackage[toc,page,title,titletoc,header]{appendix}
\usepackage{indentfirst}
\usepackage[english]{babel}
\usepackage{colortbl}

\newfloat{figure}{H}{lof}
\newfloat{table}{H}{lot}
\floatname{figure}{\figurename}
\floatname{table}{\tablename}
\newtheorem{theorem}{Theorem}[section]

\newtheorem{Assumption}[theorem]{Assumption}

\newtheorem{lemma}[theorem]{Lemma}

\newtheorem{proposition}[theorem]{Proposition}
\newtheorem{remark}[theorem]{Remark}

\newenvironment{proof}[1][Proof]{\noindent\textit{#1.} }{\hfill \rule{0.5em}{0.5em}}

\numberwithin{equation}{section}

\begin{document}

\title{\textbf{A Holling's predator-prey model with handling and searching
predators}}
\author{\textsc{Sze-Bi Hsu$^{(a)}$, Zhihua Liu$^{(b)}${\small \thanks{Research was partially supported by NSFC and CNRS (Grant Nos. 11871007 and 11811530272) and the Fundamental Research Funds for the Central Universities.}} and Pierre Magal$^{(c)}${\small \thanks{Research was partially supported by CNRS and  National Natural Science Foundation of China (Grant No.11811530272)}}}\\
{\small \textit{$^{(a)}$ Department of Mathematics and The National Center for Theoretical Science,}}\\
{\small \textit{National TsingHua University, Hsinchu 300, Taiwan  }}\\
$^{(b)}${\small \textit{School of Mathematical Sciences, Beijing Normal
University,}}\\
{\small \textit{Beijing 100875, People's Republic of China }} \\
$^{(c)}${\small \textit{Univ. Bordeaux, IMB, UMR 5251, F-33400 Talence, France.}} \\
{\small \textit{CNRS, IMB, UMR 5251, F-33400 Talence, France.}}
}
\date{}
\maketitle

\begin{abstract}
The goal of this paper is to explain how to derive the classical
Rosenzweig-MacArthur's model by using a model with two groups of predators
in which we can separate the vital dynamic and consumption of prey to
describe the behavior of the predators. This will be especially very
convenient if we want to add an age or size structure to the predator
population. As mentioned by Holling (without mathematical model), we divide
the population of predators into the searching and the handling predators.
In this article we study some properties of this model and conclude the
paper proving that the model converges to the classical
Rosenzweig-MacArthur's model by using an appropriate rescalling. We also apply this model to the Canadian snowshoe hares and lynxes.
\end{abstract}





\section{Introduction}

The article is devoted to the following predator prey system with handling
and searching predators
\begin{equation}  \label{1.1}
\left\{
\begin{array}{lll}
N^\prime & =\underset{\text{Logistic growth}}{\underbrace{
\beta_{N}N-\mu_{N}N-\delta N^2}} \underset{\text{Consumption of prey by
predators}}{\underbrace{-N \, \kappa \,P_{S}}} &  \\
P_{S}^\prime & =\underset{\text{Mortality}}{\underbrace{-(\mu_{P}+\eta)P_{S}
}} \underset{\text{Searching becoming handling}}{\underbrace{-P_{S}\, \rho
\, \kappa\, N}}+ \gamma P_{H} &  \\
P_{H}^\prime & = \underset{\text{Mortality}}{\underbrace{-\mu_{P}P_{H}}}
\hspace{1.5cm}+ P_{S}\, \rho \, \kappa\, N \underset{\text{Handling becoming
searching}}{\underbrace{-\gamma P_{H}}}+\underset{\text{New born predator}}{%
\underbrace{\beta_{P}\left( P_{S}+P_{H} \right)}} &
\end{array}
\right.
\end{equation}
where $N(t)$ is the number of prey at time $t$, $P_{S}(t)$ is the number of
predators searching for preys at time $t$, and $P_{H}(t)$ is the number of
predators handling the preys at time $t$.

Here the terminology "handling and searching predators" refers to Holling
himself \cite{Holling59b}. In the model \eqref{1.1}, the term $%
\beta_{P}\left( P_{S}(t)+P_{H}(t) \right)$ is the flux of new born
predators. Here we assume that all the new born predators are handlers. The
parameter $\rho$ should be interpreted as a conversion rate. The term $%
P_{S}(t)\, \rho \, \kappa\, N(t)$ (in the $P_S$-equation or the $P_H$%
-equation) is a flux of searching predators becoming handling predators. The
term $\gamma P_{H}(t)$ (in the $P_S$-equation or the $P_H$-equation) is the
flux of handling predators becoming searching predators. The term $\mu_{P}$
is the natural mortality of the predators and $\eta$ is an extra mortality
term for the searching predators only. The term $N(t)\, \kappa \, P_{S}(t)$
in the $N$-equation corresponds to the consumption of the preys by the
predators. The part $\beta_{N}N(t)-\mu_{N}N(t)-\delta N(t)^2$ in the $N$%
-equation is the standard logistic equation.

The main idea about this model is to distinguish the vital dynamics (birth
and death process) of the predators and their survival due to the
consumption of preys. In the model the survival of predators will depend on
the status searching or handling. The handling predators are satisfied with
their consumption of preys and they don't need to find more preys to
survive. At the opposite the searching predators are unsatisfied with their
consumption of preys and they need to find some preys to survive. Once a
searching predator finds a prey (or enough preys) he becomes a handling and
after some time the handling predator becomes a searching predator again.

This process only influences the survival of predators which depends on
their ability to find a prey. In our model, a predator can reproduce at time
$t$ because he found enough prey to survive from its birth until the time $t$%
. In section \ref{Section2} we will first make some basic assumptions in
order for the predators to extinct in absence of preys. Then based on these
setting we will analyze the dynamical properties of the system \eqref{1.1}.
The main advantage with the model \eqref{1.1} is that we can separate the
vital dynamic and consumption of preys to describe the behavior of the
predators. This will be especially very convenient if we want to add an age
or size structure to the predator population. This kind of question is left
for future work.

In section \ref{Section6} we will see that our model is also comparable to
the standard predator prey model whenever $\rho=\dfrac{\chi}{\varepsilon}$
and $\gamma=\dfrac{1}{\varepsilon}$ for $\varepsilon>0$ small which means
that predators are going back and forth from handling to searching very
rapidly. In that case (as a singular limit) we obtain a convergence result
to the standard Rosenweig-MacArthur model \cite{Rosenzweig}
\begin{equation}  \label{1.2}
\left\lbrace
\begin{array}{l}
N^{\prime }=r N \left( 1- \dfrac{N}{K}\right)- P\, \dfrac{m N}{a+N}, \\
P^{\prime }=P\left( \dfrac{m N}{a+N}-d \right),%
\end{array}
\right.
\end{equation}
which is the most popular predator-prey system discussed in the literature.

Let us recall that the derivation of Holling type II functional response $%
\frac{m N}{a+N}$ can be found in Holling \cite{Holling59a, Holling59b} and
Hsu, Hubbell and Waltman \cite{Hsu-Hubbell-Waltman78a}. There are two
mathematical problems for the system \eqref{1.2}, namely, the global
asymptotic stability of the locally asymptotically stable interior
equilibrium (when it exists) and the uniqueness of the limit cycles when the
interior equilibrium is unstable. For the global asymptotic stability of
this equilibrium we may apply the Dulac's criterion Hsu, Hubbell and Waltman
\cite{Hsu-Hubbell-Waltman78b}, weak negative Bendixson Lemma Cheng, Hsu and
Lin \cite{Cheng-Hsu-Lin} or construction Lyapunov function Ardito and
Ricciardi \cite{Ardito-Ricciardi}. For uniqueness of limit cycle of
Rosenzweig-MacArthur model \eqref{1.2}, Cheng \cite{Cheng} employed the
symmetry of the prey isocline to prove the exponential asymptotic stability
of each limit cycle. Kuang and Freedman \cite{Kuang-Freedman} reduced %
\eqref{1.2} to a generalized Lienard equation which has the uniqueness of
limit cycle Zhang \cite{Zhang}.  We refer to Murray \cite{Murray}, Hastings \cite{Hastings}, Turchin \cite{Turchin} for more results about predator prey models. 

The plan of the paper is the following. In section \ref{Section2} we set
some basic assumptions in order for the predators to extinct in absence of
preys. In section \ref{Section3} we prove that the system is dissipative. In
section \ref{Section4} we study the uniform persistence and extinction
properties of the predators. We study the system in the interior region
which corresponds the region of co-existence of preys and predators in
section \ref{Section5}. We should mention that we can obtain a rather
complete description of the asymptotic behavior thanks to the fact the
system is competitive (for a new partial order). In section \ref{Section6}
we prove the convergence of our model to the Rosenweig-MacArthur model. In section \ref{Section7} we apply the model to the Canadian snowshoe hares and lynxes.

\section{Basic assumptions}

\label{Section2} In this section, we set some basic assumptions in order for
the predators to extinct in absence of preys. Consider the total number of
predators
\begin{equation*}
P=P_{H}+P_{S}.
\end{equation*}%
Then
\begin{equation}
P^{\prime }=\left( \beta _{P}-\mu _{P}-\eta \right) P_{S}+\left( \beta
_{P}-\mu _{P}\right) P_{H}.  \label{2.2}
\end{equation}%
The following assumptions mean that when $\frac{P_{S}}{P_{H}}>-\frac{\beta
_{P}-\mu _{P}}{\beta _{P}-\mu _{P}-\eta }$, the total population of
predators decreases. The total population of predators increases otherwise.

\begin{Assumption}
\label{ASS2.1} We assume that all the parameters of the model \eqref{1.1}
are strictly positive and
\begin{equation*}
\beta_{N}-\mu_{N}>0,\, \beta_{P}-\mu_{P}>0 \text{ and } \beta_P-\mu_P-\eta
<0.
\end{equation*}
\end{Assumption}

In absence of preys the dynamics of predator population is described by
\begin{equation*}
\left\lbrace
\begin{array}{l}
P_S^{\prime }=-(\mu_{P}+\eta)P_{S}+\gamma P_{H} \\
P_H^{\prime }=\beta_{P}P_S+\left(\beta_{P}-\mu_{P}-\gamma \right) P_{H}.%
\end{array}
\right.
\end{equation*}
Define
\begin{equation}  \label{2.3}
M= \left[
\begin{array}{cc}
-\mu_{P}-\eta & \gamma \\
\beta_{P} & \beta_{P}-\mu_{P}-\gamma%
\end{array}
\right].
\end{equation}
By using Assumption \ref{ASS2.1} we have
\begin{equation*}
\mathrm{tr}\left( M\right)= \left( \beta_{P}-\mu_{P}-\gamma \right)
-\left(\mu_{P}+\eta \right)<0.
\end{equation*}
Therefore in absence of preys the population of predators goes to extinct if
and only if
\begin{equation*}
\mathrm{det}\left( M\right)= - \left(\mu_{P}+\eta \right)\left(
\beta_{P}-\mu_{P}-\gamma \right) - \beta_{P} \gamma>0.
\end{equation*}
This last inequality can be equivalently reformulated in the following
assumption.

\begin{Assumption}[Extinction of the predators]
\label{ASS2.2} We assume that
\begin{equation}  \label{2.4}
\left( \beta_P-\mu_P-\gamma \right) <-\dfrac{\beta_P \gamma}{\mu_P+\eta}
\Leftrightarrow \left( \beta_P-\mu_P\right) <-\dfrac{\gamma}{\mu_P+\eta}
\left(\beta_P-\mu_P-\eta \right) .
\end{equation}
\end{Assumption}

\begin{remark}
\label{REM2.3} The first inequality in \eqref{2.4} implies that $\left(
\beta_P-\mu_P-\gamma \right)<0$. Moreover the second inequality in %
\eqref{2.4} and $\left( \beta_P-\mu_P\right)>0$ imply that $%
\left(\beta_P-\mu_P-\eta \right)<0$.
\end{remark}

\begin{lemma}
\label{LE2.4} Let Assumptions \ref{ASS2.1} and \ref{ASS2.1} be satisfied.
Then in absence of preys the population of predators goes to extinct.
\end{lemma}

\section{Dissipativity}

\label{Section3} In this section, we will prove that the system \eqref{1.1}
is dissipative. We look for a positive left eigen-vector $(\widetilde{P}_S,%
\widetilde{P}_H) \in (0,+\infty)^2$ and an eigenvalue $\lambda>0$ such that
\begin{equation*}
(\widetilde{P}_S,\widetilde{P}_H) \left[
\begin{array}{cc}
-\mu_{P}-\eta & \gamma \\
\beta_{P} & \beta_{P}-\mu_{P}-\gamma%
\end{array}
\right]= -\lambda (\widetilde{P}_S,\widetilde{P}_H)
\end{equation*}
that is equivalent to
\begin{equation*}
\left\lbrace
\begin{array}{l}
-\left( \mu_{P}+\eta \right)\widetilde{P}_S+\beta_{P} \widetilde{P}%
_H=-\lambda \widetilde{P}_S \\
\gamma \widetilde{P}_S+ \left( \beta_{P}-\mu_{P}-\gamma \right) \widetilde{P}%
_H=-\lambda \widetilde{P}_H%
\end{array}
\right. \Leftrightarrow \left\lbrace
\begin{array}{l}
\beta_{P} \widetilde{P}_H=\left[ \left(\mu_{P}+\eta \right) -\lambda \right]%
\widetilde{P}_S \\
\gamma \widetilde{P}_S=\left[- \left(\beta_{P}-\mu_{P}-\gamma \right)
-\lambda \right] \widetilde{P}_H.%
\end{array}
\right.
\end{equation*}
Thus the sign $\widetilde{P}_S$ and $\widetilde{P}_H$ are the same if we
impose
\begin{equation*}
\lambda \in \left( 0, \min \left( \left(\mu_{P}+\eta \right), -
\left(\beta_{P}-\mu_{P}-\gamma \right) \right) \right)
\end{equation*}
and $\lambda$ must satisfy the following equation
\begin{equation*}
1=\dfrac{ \left[\left(\mu_{P}+\eta \right) -\lambda \right]}{\beta_{P}} \;
\dfrac{\left[ - \left(\beta_{P}-\mu_{P}-\gamma \right) -\lambda \right]}{%
\gamma} =:\Psi\left( \lambda \right).
\end{equation*}
The function $\lambda \to \Psi\left( \lambda \right)$ decreases between $0$
and $\min \left( \left(\mu_{P}+\eta \right), -
\left(\beta_{P}-\mu_{P}-\gamma \right) \right)$ and by using \eqref{2.4} we
have $\Psi(0)>1$. It follows that there exists a unique $\lambda^* \in
\left( 0, \min \left( \left(\mu_{P}+\eta \right), -
\left(\beta_{P}-\mu_{P}-\gamma \right) \right) \right)$ such that
\begin{equation}  \label{3.1}
1=\dfrac{ \left[\left(\mu_{P}+\eta \right) -\lambda^* \right]}{\beta_{P}} \;
\dfrac{\left[ - \left(\beta_{P}-\mu_{P}-\gamma \right) -\lambda^* \right]}{%
\gamma}.
\end{equation}
Note that
\begin{equation*}
\dfrac{\left[ - \left(\beta_{P}-\mu_{P}-\gamma \right) -\lambda^* \right]}{%
\gamma} <1 \Leftrightarrow - \left(\beta_{P}-\mu_{P} \right)<\lambda^*.
\end{equation*}%
By assumption $\left(\beta_{P}-\mu_{P} \right)>0$ it follows from %
\eqref{3.1} that
\begin{equation*}
\dfrac{ \left[\left(\mu_{P}+\eta \right) -\lambda^* \right]}{\beta_{P}}>1.
\end{equation*}
Since
\begin{equation*}
\gamma \widetilde{P}_S=\left[- \left(\beta_{P}-\mu_{P}-\gamma \right)
-\lambda^* \right] \widetilde{P}_H,
\end{equation*}
it follows that
\begin{equation}  \label{3.2}
\widetilde{P}_H>\widetilde{P}_S>0.
\end{equation}
By using $P_S$-equation and $P_H$-equation of system \eqref{1.1} we obtain
\begin{equation}  \label{3.3}
\widetilde{P}_S P_{S}^{\prime }+\widetilde{P}_H P_{H}^{\prime }=-\lambda^*%
\left[ \widetilde{P}_S P_S + \widetilde{P}_H P_H \right]- \left(\widetilde{P}%
_S-\widetilde{P}_H \right) P_{S} \, \rho \, \kappa\, N.
\end{equation}
By using the $N$-equation and comparison principle it is clear that we can
find some $N^*=\max\left(N_0, \left(\beta_{N}-\mu_{N} \right)/\delta \right)$
such that
\begin{equation*}
N(t) \leq N^*, \forall t \geq 0,
\end{equation*}
Then it follows that
\begin{equation*}
\rho \left(\widetilde{P}_H-\widetilde{P}_S \right) N^{\prime }+\widetilde{P}%
_S P_{S}^{\prime }+\widetilde{P}_H P_{H}^{\prime }\leq -\rho \left(%
\widetilde{P}_H-\widetilde{P}_S \right)\mu_{N}N -\lambda^* \left[ \widetilde{%
P}_S P_S + \widetilde{P}_H P_H \right] + \rho \left(\widetilde{P}_H-%
\widetilde{P}_S \right)\beta_{N}N^*
\end{equation*}
and the dissipativity follows.

Set
\begin{equation*}
M=\dfrac{\rho \left(\widetilde{P}_H-\widetilde{P}_S \right)\beta_N N^\ast }{%
\min \left( \mu_N, \lambda^\ast\right)}>0.
\end{equation*}
As a consequence of the last inequality, we obtain the following results.

\begin{proposition}
Let Assumptions \ref{ASS2.1} and \ref{ASS2.2} be satisfied. The system %
\eqref{1.1} generates a unique continuous semiflow $\left\lbrace U(t)
\right\rbrace_{t \geq 0}$ on $[0,\infty)^3$. Moreover the domain
\begin{equation*}
D=\left\lbrace \left(N, P_S, P_H \right) \in [0,\infty)^3: \rho \left(%
\widetilde{P}_H-\widetilde{P}_S \right) N+ \widetilde{P}_S P_S+\widetilde{P}
_H P_H \leq M \right\rbrace
\end{equation*}
is positively invariant by the semiflow generated by $U$. That is to say
that
\begin{equation*}
U(t)D \subset D , \forall t \geq 0.
\end{equation*}
Furthermore $D$ attracts every point of $[0,\infty)^3$ for $U$. That is to
say that
\begin{equation*}
\lim_{t \to \infty} \delta(U(t)x,D)=0, \forall x \in [0,\infty)^3,
\end{equation*}
where $\delta(x,D):=\inf_{y \in D} \Vert x -y \Vert$ is the Hausdorff's
semi-distance. As a consequence the semiflow of $U$ has a compact global
attractor $\mathcal{A} \subset [0,\infty)^3$.
\end{proposition}

\section{Uniform persitence and extinction of predators}

In this section, we study the uniform persistence and extinction of the
predators. Firstly we consider the existence of the equilibrium. \label%
{Section4} The equilibrium $\left(\overline{N},\overline{P}_{S}, \overline{P}%
_{H} \right) \in [0, \infty)^3$ satisfies the following system
\begin{equation*}
\left\{
\begin{array}{l}
0=\overline{N} \left[ \beta _{N}-\mu _{N}-\delta \overline{N}-\,\kappa \,%
\overline{P}_{S}\right] , \\
0=-(\mu _{P}+\eta )\overline{P}_{S}-\overline{P}_{S}\,\rho \,\kappa \,%
\overline{N}+\gamma \overline{P}_{H}, \\
0=-\mu _{P}\overline{P}_{H}+\overline{P}_{S}\,\rho \,\kappa \,\overline{N}%
-\gamma \overline{P}_{H}+\beta _{P}\left( \overline{P}_{S}+\overline{P}%
_{H}\right).%
\end{array}%
\right.
\end{equation*}
By using Assumptions \ref{ASS2.1} and \ref{ASS2.2}, we deduce that the only
equilibrium satisfying $\overline{N}=0$ is $E_{1}=(0,0,0)$. If we assume
next that $\overline{N}>0$, we obtain the system
\begin{equation*}
\left\{
\begin{array}{l}
0= \beta _{N}-\mu _{N}-\delta \overline{N}-\,\kappa \,\overline{P}_{S} , \\
0=-(\mu _{P}+\eta )\overline{P}_{S}-\overline{P}_{S}\,\rho \,\kappa \,%
\overline{N}+\gamma \overline{P}_{H}, \\
0=-\mu _{P}\overline{P}_{H}+\overline{P}_{S}\,\rho \,\kappa \,\overline{N}%
-\gamma \overline{P}_{H}+\beta _{P}\left( \overline{P}_{S}+\overline{P}%
_{H}\right).%
\end{array}%
\right.
\end{equation*}
From the first equation we have
\begin{equation*}
\overline{N}=\widehat{N}-\frac{\,\kappa \,}{\delta }\overline{P}_{S}
\end{equation*}%
with
\begin{equation*}
\widehat{N}=\dfrac{\beta_N-\mu_N}{\delta}.
\end{equation*}
By adding the last two equations, we have
\begin{equation*}
\overline{P}_{H}=\frac{(\mu _{P}+\eta -\beta _{P})}{\left( \beta _{P}-\mu
_{P}\right) }\overline{P}_{S}.
\end{equation*}%
Combining the above two equations with

\begin{equation*}
-(\mu _{P}+\eta )\overline{P}_{S}-\overline{P}_{S}\,\rho \,\kappa \,%
\overline{N}+\gamma \overline{P}_{H}=0,
\end{equation*}%
we have
\begin{equation*}
\left( -(\mu _{P}+\eta )-\,\frac{\left( \beta _{N}-\mu _{N}\right) \rho
\,\kappa }{\delta }+\gamma \frac{(\mu _{P}+\eta -\beta _{P})}{\left( \beta
_{P}-\mu _{P}\right) }\right) \overline{P}_{S}\,+\frac{\,\kappa^{2} \,\rho
\, }{\delta } \overline{P}_{S}^{2}\,=0
\end{equation*}%
and then%
\begin{equation*}
\overline{P}_{S}=0\text{ or }\overline{P}_{S}=\left( (\mu _{P}+\eta )\,-%
\frac{\gamma (\mu _{P}+\eta -\beta _{P})}{\left( \beta _{P}-\mu _{P}\right) }%
\right) \frac{\,\delta }{\kappa ^{2}\,\rho \,}+\,\frac{\left( \beta _{N}-\mu
_{N}\right) }{\kappa }.
\end{equation*}%
Thus we get the following lemma.

\begin{lemma}
\label{LE4.1} Let Assumptions \ref{ASS2.1} and \ref{ASS2.2} be satisfied.
System \eqref{1.1} always has the following two boundary equilibria
\begin{equation*}
E_{1}=(0,0,0),\quad E_{2}=\left( \widehat{N} ,0,0\right).
\end{equation*}
Moreover there exists a unique interior equilibrium $\quad E^{\ast }=\left( N^{\ast
},P_{S}^{\ast },P_{H}^{\ast }\right)$ if and only if
\begin{equation}  \label{4.1}
\left( \beta _{N}-\mu _{N}\right) \left( \beta _{P}-\mu _{P}\right) \kappa
\,\rho+\delta \, \left( \beta _{P}-\mu _{P}\right) \left( \mu _{P}+\eta
\right) >-\delta \gamma (\beta _{P}-\mu _{P}-\eta ).
\end{equation}
Furthermore, we have
\begin{equation*}
\begin{array}{ll}
N^{\ast } & =\dfrac{-\,\left( \beta _{P}-\mu _{P}\right) \left( \mu
_{P}+\eta \right) -\gamma (\beta _{P}-\mu _{P}-\eta )}{\left( \beta _{P}-\mu
_{P}\right) \kappa \,\rho \,}>0, \\
P_{S}^{\ast } & =\dfrac{\delta \,\left( \beta _{P}-\mu _{P}\right) \left(
\mu _{P}+\eta \right) +\delta \gamma (\beta _{P}-\mu _{P}-\eta )\,+\left(
\beta _{N}-\mu _{N}\right) \left( \beta _{P}-\mu _{P}\right) \kappa \,\rho \,%
}{\left( \beta _{P}-\mu _{P}\right) \kappa ^{2}\,\rho \,}>0, \\
P_{H}^{\ast } & =-\dfrac{\left(\beta _{P}-\mu _{P}-\eta \right)}{\beta
_{P}-\mu _{P}} P_{S}^{\ast }>0.%
\end{array}%
\end{equation*}
\end{lemma}

%

\subsection{Stability of the equilibrium $E_{1}$}

The Jacobian matrix at the equilibrium $E_{1}$ is
\begin{equation*}
\left[
\begin{array}{ccc}
\beta _{N}-\mu _{N} & 0 & 0 \\
0 & -\mu _{P}-\eta & \gamma \\
0 & \beta _{P} & \beta _{P}-\mu _{P}-\gamma%
\end{array}%
\right]
\end{equation*}
and the characteristic equation is
\begin{equation*}
\left[ \left( \lambda +\mu _{P}+\eta \right) \left( \lambda -\left( \beta
_{P}-\mu _{P}-\gamma \right) \right) -\beta _{P}\gamma \right] \left[
\lambda -\left( \beta _{N}-\mu _{N}\right) \right] =0.
\end{equation*}
So one of the eigenvalues is $\lambda _{1,E_{1}}=\beta _{N}-\mu _{N}>0.$
Thus we can get that the equilibrium $E_{1}$ is unstable. The rest of the
spectrum coincides with the spectrum of the matrix $M$ defined in \eqref{2.3}%
. Thus we obtain the following lemma.

\begin{lemma}
\label{LE4.2} Let Assumptions \ref{ASS2.1} and \ref{ASS2.2} be satisfied.
The equilibrium $E_{1}$ is hyperbolic and the unstable space is one
dimensional.
\end{lemma}

\subsection{Stability of the equilibrium $E_2$}

The Jacobian matrix at the equilibrium $E_{2}$ is
\begin{equation*}
\left[
\begin{array}{ccc}
-\left( \beta _{N}-\mu _{N}\right) & -\kappa \widehat{N} & 0 \\
0 & -\left( (\mu _{P}+\eta )+\rho \kappa \widehat{N} \right) & \gamma \\
0 & \rho \,\kappa \,\widehat{N}+\beta _{P} & \beta _{P}-\mu _{P}-\gamma%
\end{array}
\right]
\end{equation*}%
and the characteristic equation is
\begin{equation*}
\left[ \left( \lambda + \mu _{P}+\eta +\rho \kappa \widehat{N} \right)
\left( \lambda -(\beta _{P}-\mu _{P}-\gamma )\right) -\gamma \left( \rho
\kappa \widehat{N}+\beta _{P}\right) \right] \left[ \lambda +\left( \beta
_{N}-\mu _{N}\right) \right] =0.
\end{equation*}%
So one of the eigenvalues is $\lambda _{1,E_{2}}=-\left( \beta _{N}-\mu
_{N}\right) <0$ and the remaining part of the characteristic equation is
\begin{equation*}
\lambda^2 +a \lambda +b=0
\end{equation*}
with
\begin{equation*}
a=\left( \mu _{P}+\eta +\,\rho \,\kappa \,\widehat{N}\right) -(\beta
_{P}-\mu _{P}-\gamma )
\end{equation*}
and
\begin{eqnarray*}
b &=&\left( \mu _{P}+\gamma -\beta _{P}\right)\left( \mu _{P}+\eta +\,\rho
\,\kappa \,\widehat{N}\right) -\gamma \left( \rho \,\kappa \,\widehat{N}%
+\beta _{P}\right) .
\end{eqnarray*}
By using Assumptions \ref{ASS2.1} and \ref{ASS2.2} we have $a>0$. Moreover
by using the Routh-Hurwitz criterion $E_2$ is stable if and only if $b>0$
which corresponds to
\begin{equation*}
\begin{array}{l}
\left( \mu _{P}+\gamma -\beta _{P}\right) \left( \mu _{P}+\eta +\rho \kappa
\widehat{N}\right) -\gamma \left( \rho \,\kappa \,\widehat{N}+\beta
_{P}\right) >0 \\
\Leftrightarrow \left( \beta _{N}-\mu _{N}\right) \left( \beta _{P}-\mu
_{P}\right) \kappa \,\rho+\delta \, \left( \beta _{P}-\mu _{P}\right) \left(
\mu _{P}+\eta \right) <-\delta \gamma (\beta _{P}-\mu _{P}-\eta )%
\end{array}%
\end{equation*}
Now we obtain the following result.

\begin{lemma}
\label{LE4.3} Let Assumptions \ref{ASS2.1} and \ref{ASS2.2} be satisfied. $%
E_2$ is unstable if the interior equilibrium exits (i.e. the condition \ref%
{4.1} is satisfied) and the unstable space is one dimensional and the stable
space is two dimensional.
\end{lemma}

\subsection{Extinction of the predators and the global stability of $E_2$}

We decompose the positive cone $M=\mathbb{R}^3_+$ into the interior region
\begin{equation*}
\overset{\circ}{M}= \left\lbrace \left(N,P_S,P_H\right) \in M: N>0 \text{
and } P_S+P_H >0 \right\rbrace,
\end{equation*}
the boundary region with predators only
\begin{equation}  \label{4.2}
\partial M_P:= \left\lbrace \left(N,P_S,P_H\right) \in M: N =0 \right\rbrace,
\end{equation}
and the boundary region with prey only
\begin{equation}  \label{4.3}
\partial M_N:= \left\lbrace \left(N,P_S,P_H\right) \in M: P_S+P_H =0
\right\rbrace.
\end{equation}
Each sub domain $\overset{\circ}{M}$, $\partial M_P$ and $\partial M_N$ is
positively invariant by the semiflow generated by \eqref{1.1}.

\begin{theorem}
\label{TH4.4} Let Assumptions \ref{ASS2.1} and \ref{ASS2.2} be satisfied.
Assume that $E_2$ is locally asymptotically stable (i.e. $\left(\mu
_{P}+\eta+\rho \kappa \widehat{N} \right) \left(\mu _{P}+ \gamma- \beta_P
\right) > \gamma \left( \beta_P+\rho \kappa \widehat{N} \right)$). Then the
predator goes to extinction. More precisely for each initial value in $%
M=(N(0),P_S(0),P_H(0)) \in [0,\infty)^3$,
\begin{equation*}
\lim_{t \to \infty} P_S(t)+P_H(t)=0.
\end{equation*}
and
\begin{equation*}
\lim_{t \to \infty} N(t)= \left\lbrace
\begin{array}{ll}
\widehat{N}, & \text{ if } N(0)>0, \\
0, & \text{ if } N(0)=0.%
\end{array}
\right.
\end{equation*}
\end{theorem}

\begin{proof}
The boundary region with predator only $\partial M_P$ is positively
invariant by the semiflow generated by \eqref{1.1} and by Assumption \ref%
{ASS2.2} any solution starting from $\partial M_P$ exponentially converges
to $E_{1}$.

So it remains to investigate the limit of a solution starting from $\overset{%
\circ}{M} \cup \partial M_N \backslash \{E_{1}\}$. We consider the Liapunov
function
\begin{equation}  \label{4.4}
V(N,P_S,P_H)= \int_{\widehat{N}}^N \dfrac{\xi-\widehat{N}}{\xi} d\xi+c_1
P_S+ c_2 P_H
\end{equation}
where $c_1>0$ and $c_2>0$ to be determined. We have
\begin{equation}
\begin{array}{l}
\dot{V}= \left( N- \widehat{N}\right) \left(\beta_{N}-\mu_{N}-\delta N-
\kappa P_{S} \right) \\
+c_1 \left(-(\mu _{P}+\eta )P_{S}-\rho \kappa P_{S} N+\gamma P_{H}\right) \\
+c_2 \left( -\mu _{P}P_{H}+\rho \kappa P_{S}N-\gamma P_{H}+\beta_{P}\left(
P_{S}+P_{H} \right) \right) \\
=\left( N- \widehat{N}\right) \left(-\delta \left( N-\widehat{N} \right)-
\kappa P_{S} \right) \\
+c_1 \left(-(\mu _{P}+\eta )P_{S}-\rho \kappa P_{S} \left( N -\widehat{N}
\right)-\rho \kappa P_{S}\widehat{N} +\gamma P_{H}\right) \\
+c_2 \left( -\mu _{P}P_{H}+\rho \kappa P_{S} \left( N -\widehat{N}
\right)+\rho \kappa P_{S}\widehat{N} -\gamma P_{H}+\beta_{P}\left(
P_{S}+P_{H} \right) \right).%
\end{array}%
\end{equation}
Thus we obtain
\begin{equation}
\begin{array}{l}
\dot{V}= -\delta\left( N- \widehat{N}\right)^2+ \kappa P_{S}\left( N-
\widehat{N}\right) \left(-1-c_1\rho+c_2 \rho\right) \\
+P_S \left(- c_1 (\mu _{P}+\eta )- c_1 \rho \kappa \widehat{N} + c_2 \rho
\kappa \widehat{N}+c_2 \beta_P \right) \\
+P_H \left(- c_2 \mu _{P}+ c_1 \gamma - c_2 \gamma + c_2 \beta_P \right).%
\end{array}%
\end{equation}
We claim that we can choose $c_1>0$ and $c_2>0$ such that $c_2=c_1+\dfrac{1}{%
\rho}$ and the following inequalities are satisfied
\begin{equation}  \label{4.8}
- c_1 (\mu _{P}+\eta )- c_1 \rho \kappa \widehat{N} + c_2 \rho \kappa
\widehat{N}+c_2 \beta_P <0 \text{ and } - c_2 \mu _{P}+ c_1 \gamma - c_2
\gamma + c_2 \beta_P<0.
\end{equation}
In fact the inequalities in \eqref{4.8} lead to consider the lines
\begin{equation*}
c_2=c_1 \dfrac{\gamma}{\mu _{P}+ \gamma- \beta_P} \,\, (L_1)
\end{equation*}
and
\begin{equation*}
c_2=c_1 \dfrac{(\mu _{P}+\eta )+\rho \kappa \widehat{N}}{\beta_P+\rho \kappa
\widehat{N}} \,\,(L_2).
\end{equation*}
By Assumption \ref{ASS2.2} (see Remark \ref{REM2.3}) we have $\mu _{P}+
\gamma- \beta_P>0$ and by Assumption \ref{ASS2.1} we have $\mu _{P}+
\eta>\beta_P$ and then
\begin{equation*}
\dfrac{(\mu _{P}+\eta )+\rho \kappa \widehat{N}}{\beta_P+\rho \kappa
\widehat{N}}>1.
\end{equation*}
Note that
\begin{equation*}
\begin{array}{l}
\dfrac{(\mu _{P}+\eta )+\rho \kappa \widehat{N}}{\beta_P+\rho \kappa
\widehat{N}}> \dfrac{\gamma}{\mu _{P}+ \gamma- \beta_P} \\
\Leftrightarrow \left(\mu _{P}+\eta+\rho \kappa \widehat{N} \right)
\left(\mu _{P}+ \gamma- \beta_P \right) > \gamma \left( \beta_P+\rho \kappa
\widehat{N} \right) \\
\end{array}%
\end{equation*}
and thus we obtain that the slope of $L_2$ is greater than the slope of $L_1$%
. Finally we have
\begin{equation*}
\lim_{N \to 0^+} V(N,P_S,P_N)=(N-\widehat{N})-\widehat{N} \ln\left( \dfrac{N%
}{\widehat{N}} \right)+c_1P_S+c_2P_N =+\infty.
\end{equation*}
By LaSalle's invariance principle we obtain that $E_2$ is globally
asymptotically stable for the system restricted to $\overset{\circ}{M} \cup
\partial M_N \backslash \{E_{1}\}$.
\end{proof}

\subsection{Uniform persistence of the predators}

We decompose the positive cone into
\begin{equation*}
\mathbb{R}_+^3=\partial M \cup \overset{\circ}{M}
\end{equation*}
where the boundary region is defined as
\begin{equation*}
\partial M := \partial M_P \cup \partial M_N.
\end{equation*}
It is clear that both regions $\overset{\circ}{M}$ and $\partial M$ are
positively invariant by the semiflow generated by the system. Moreover we
have the following result.

\begin{theorem}
Let Assumptions \ref{ASS2.1} and \ref{ASS2.2} be satisfied. If the interior
equilibrium exits then the predators uniformly persist with respect to the
domain decomposition $\left( \partial M,\overset{\circ }{M}\right) $. That
is to say that there exists $\varepsilon >0$ such that for each initial
value $N(0)>0$ and $P_{S}(0)+P_{H}(0)>0$
\begin{equation*}
\liminf_{t\rightarrow \infty }N(t)>\varepsilon \text{ and }%
\liminf_{t\rightarrow \infty }P_{S}(t)+P_{H}(t)>\varepsilon .
\end{equation*}
\end{theorem}

\begin{proof}
The equilibrium $E_{1}=\{(0,0,0)\}$ is clearly chained to $E_{2}=\{(\widehat{%
N},0,0)\}$. By using Theorem 4.1 in \cite{Hale-Waltman}, we only need to
prove that
\begin{equation*}
W^{s}(E_{i})\cap \overset{\circ }{M}=\varnothing ,
\end{equation*}%
where $i=1,2$ and
\begin{equation*}
W^{s}(E_{i})=\left\{ \left( N,P_{S},P_{H}\right) \in M:\omega (\left(
N,P_{S},P_{H}\right) )\neq \varnothing \text{ and }\omega (\left(
N,P_{S},P_{H}\right) )\subset E_{i}\right\} .
\end{equation*}%
Assume that there exists $E^{0}=\left( N^{0},P_{S}^{0},P_{H}^{0}\right) \in
\overset{\circ }{M}$ (which means $N^{0}>0$ and $P_{S}^{0}+P_{H}^{0}>0$)
such that $\omega (E^{0})\subset E_{1}.$ Then for any $\varepsilon >0,$
there exists $t_{0}\geq 0,$ such that
\begin{equation*}
N(t)+P_{S}(t)+P_{H}(t)\leq \varepsilon ,\forall t\geq t_{0}
\end{equation*}%
where $(N(t),P_{S}(t),P_{H}(t))=U(t)E^{0}.$ By using the first equation of
model \eqref{1.1}
\begin{equation*}
N^{\prime }=\beta _{N}N-\mu _{N}N-\delta N^{2}-N\,\kappa \,P_{S},
\end{equation*}%
we have
\begin{equation*}
N^{\prime }\geq N\left( \beta _{N}-\mu _{N}-\delta \varepsilon -\kappa
\,\varepsilon \right) .
\end{equation*}%
Therefore for $\varepsilon >0$ small enough, we have $\beta _{N}-\mu
_{N}-\delta \varepsilon -\kappa \,\varepsilon >0$ and then
\begin{equation*}
\lim_{t\rightarrow \infty }N(t)=\infty
\end{equation*}%
which is in contradiction to the dissipativity of the model. Assume that
there exists $E^{0}=\left( N^{0},P_{S}^{0},P_{H}^{0}\right) \in \overset{%
\circ }{M}$ such that $\omega (E^{0})\subset E_{2}.$ Then for any $%
\varepsilon >0,$ there exists $t_{0}\geq 0,$ such that
\begin{equation*}
\left\vert N(t)-\widehat{N}\right\vert +P_{S}(t)+P_{H}(t)\leq \varepsilon
,\forall t\geq t_{0}
\end{equation*}%
where $(N(t),P_{S}(t),P_{H}(t))=U(t)E^{0}$. By using the two last equation
of system \eqref{1.1}, we obtain
\begin{equation}
\begin{array}{ll}
P_{S}^{\prime }\geq -(\mu _{P}+\eta )P_{S}-P_{S}\rho \kappa \,\left(
\widehat{N}+\varepsilon \right) +\gamma P_{H} &  \\
P_{H}^{\prime }\geq -\mu _{P}P_{H}+P_{S}\,\rho \,\kappa \,\left( \widehat{N}%
-\varepsilon \right) -\gamma P_{H}+\beta _{P}\left( P_{S}+P_{H}\right)  &
\end{array}
\label{4.9}
\end{equation}%
By using the fact that for $\varepsilon >0$ small enough the right hand side
of \eqref{4.9} is a cooperative system together with Lemma \ref{LE4.3} we
deduce that
\begin{equation*}
\lim_{t\rightarrow \infty }P_{S}(t)+P_{H}(t)=\infty .
\end{equation*}%
This gives a contradiction with the dissipativity of the system. Therefore
the uniform persistence follows.
\end{proof}

As a consequence of the dissipativity as well as the uniform peristence (see
Magal and Zhao \cite{Magal-Zhao}) we deduce the following result.

\begin{theorem}
Let Assumptions \ref{ASS2.1} and \ref{ASS2.2} be satisfied. Assume in
addition that the interior equilibrium exits. Then the system \eqref{1.1}
has a global attractor $A_0$ in the interior region $\overset{\circ}{M}$.
Namely $A_0$ is a compact invariant set by the semiflow generated by %
\eqref{1.1} on $\overset{\circ}{M}$ and $A_0$ is locally stable and attracts
the compact subsets of $\overset{\circ}{M}$.
\end{theorem}

\section{Interior region}
\label{Section5}
In this section, we will study the system in the interior region which
corresponds to the region of co-existence of preys and predators.

\subsection{Local stability of $E^{\ast }$}

The Jacobian matrix at the equilibrium $E^{\ast }$ is
\begin{equation*}
\left[
\begin{array}{ccc}
\left( \beta _{N}-\mu _{N}-2\delta N^{\ast }-\,\kappa \,P_{S}^{\ast }\right)
& -N^{\ast }\,\kappa & 0 \\
-P_{S}^{\ast }\,\rho \,\kappa & -\left( \mu _{P}+\eta +\,\rho \,\kappa
\,N^{\ast }\right) & \gamma \\
P_{S}^{\ast }\,\rho \,\kappa & \,\rho \,\kappa \,N^{\ast }\,+\beta _{P} &
\beta _{P}-\mu _{P}-\gamma%
\end{array}%
\right] .
\end{equation*}%
and the characteristic equation is

\bigskip

\begin{equation*}
\lambda ^{3}+p_{1}\lambda ^{2}+p_{2}\lambda +p_{3}=0
\end{equation*}%
with%
\begin{eqnarray*}
p_{1} &=&-\left( \beta _{N}-\mu _{N}-2\delta N^{\ast }-\,\kappa
\,P_{S}^{\ast }\right) +\left( \mu _{P}+\eta +\,\rho \,\kappa \,N^{\ast
}\right) -\left( \beta _{P}-\mu _{P}-\gamma \right),  \\
p_{2} &=&-\left( \mu _{P}+\eta +\,\rho \,\kappa \,N^{\ast }\right) \left(
\beta _{N}-\mu _{N}-2\delta N^{\ast }-\,\kappa \,P_{S}^{\ast }\right)
-N^{\ast }\,\kappa P_{S}^{\ast }\,\rho \,\kappa  \\
&&+\left( \beta _{N}-\mu _{N}-2\delta N^{\ast }-\,\kappa \,P_{S}^{\ast
}\right) \left( \beta _{P}-\mu _{P}-\gamma \right)  \\
&&-\left( \mu _{P}+\eta +\,\rho \,\kappa \,N^{\ast }\right) \left( \beta
_{P}-\mu _{P}-\gamma \right) -\left( \,\rho \,\kappa \,N^{\ast }\,+\beta
_{P}\right) \gamma,  \\
p_{3} &=&\left( \beta _{N}-\mu _{N}-2\delta N^{\ast }-\,\kappa \,P_{S}^{\ast
}\right) \left( \mu _{P}+\eta +\,\rho \,\kappa \,N^{\ast }\right) \left(
\beta _{P}-\mu _{P}-\gamma \right)  \\
&&+N^{\ast }\,\kappa \gamma P_{S}^{\ast }\,\rho \,\kappa +\gamma \left(
\beta _{N}-\mu _{N}-2\delta N^{\ast }-\,\kappa \,P_{S}^{\ast }\right) \left(
\,\rho \,\kappa \,N^{\ast }\,+\beta _{P}\right) + \\
&&P_{S}^{\ast }\,\rho \,\kappa N^{\ast }\,\kappa \left( \beta _{P}-\mu
_{P}-\gamma \right) .
\end{eqnarray*}%
By using Routh-Hurwitz criterion, we get that the equilibrium $E^{\ast }$ is
stable if and only if
\begin{equation*}
p_{1}>0,p_{1}p_{2}-p_{3}>0\text{ and }p_{3}>0.
\end{equation*}%
By computing, we have
\begin{eqnarray*}
p_{1} &=&\frac{-\kappa \,\rho \left( \beta _{P}-\mu _{P}-\gamma \right)
\left( \beta _{P}-\mu _{P}\right) -\gamma \left( \delta +\kappa \,\rho
\right) \,(\beta _{P}-\mu _{P}-\eta )-\delta \left( \beta _{P}-\mu
_{P}\right) \,\left( \mu _{P}+\eta \right) }{\kappa \,\rho \left( \beta
_{P}-\mu _{P}\right) \,}, \\
p_{2} &=&\frac{\left[ \,\left( \beta _{P}-\mu _{P}-\gamma \right) \left( \mu
_{P}+\eta \right) +\gamma \beta _{P}\right] \left\{ \left( \beta _{P}-\mu
_{P}\right) \left[ \delta \left( \beta _{P}+\eta +\gamma \right) +\kappa
\,\rho \left( \beta _{N}-\mu _{N}\right) \right] -2\delta \gamma \eta
\right\} \,}{\kappa \,\rho \left( \beta _{P}-\mu _{P}\right) ^{2}}, \\
p_{3} &=&\frac{\,\left[ \left( \beta _{P}-\mu _{P}-\gamma \right) \left( \mu
_{P}+\eta \right) +\gamma \beta _{P}\right] \left\{
\begin{array}{c}
-\delta \left( \beta _{P}-\mu _{P}-\gamma \right) \left( \mu _{P}+\eta
\right) -\gamma \delta \beta _{P}- \\
\,\kappa \,\rho \left( \beta _{P}-\mu _{P}\right) \left( \beta _{N}-\mu
_{N}\right)
\end{array}%
\right\} \,}{\kappa \,\rho \left( \beta _{P}-\mu _{P}\right) }.
\end{eqnarray*}

Thus we have the following result.
\begin{lemma}
\label{LE5.2} Let Assumptions \ref{ASS2.1}, \ref{ASS2.2} and inequality \ref{4.1} be satisfied. The equilibrium $E^{\ast }$ is stable if and only
if $\left( \beta _{P}-\mu _{P}\right) \left[ \kappa \,\rho \left( \beta
_{N}-\mu _{N}\right) +\delta \left( \eta +\gamma \right) \right] <\delta %
\left[ 2\gamma \eta -\beta _{P}\left( \beta _{P}-\mu _{P}\right) \right]$.
\end{lemma}

\subsection{Three dimensional $\mathbb{K}$-competitive system}

 In this section we use a Poincar\'e-Bendixson theorem for
three dimensional $\mathbb{K}$-competitive system (see Smith \cite[Theorem
4.2 p. 43]{Smith}). By applying this theorem to the system \eqref{1.1}
restricted to the interior global attractor $A_0$ we obtain the following
result.

\begin{theorem}
Suppose that $E^\ast=\left( N^\ast, P_S^\ast, P_H^\ast \right)$ exists and
is hyperbolic and unstable for \eqref{1.1}. Then the stable manifold $W^s\left(E^\ast \right)$ of $%
E^\ast$ is one dimensional and the omega limit set
$\omega \left( N(0),P_S(0),P_H(0) \right) $ is a nontrivial periodic orbit
in $\mathbb{R}_+^3$ for every $\left( N(0),P_S(0),P_H(0) \right) \in \mathbb{%
R}_+^3 \setminus W^s\left(E^\ast \right)$.
\end{theorem}

\begin{proof}
The Jacobian matrix of the vector field \eqref{1.1} at the point $(N,P_S,
P_H ) \in (0,\infty)^3$ is given by
\begin{equation}
J= \left(
\begin{array}{ccc}
(\beta_N-\mu_N)-2 \delta N-\kappa P_S & -\kappa N & 0 \\
- \rho \kappa P_S & -\left(\mu_P+\eta \right)- \rho \kappa N & \gamma \\
\rho \kappa P_S & \rho \kappa N+ \beta_P & -\mu_P+\beta_P-\gamma%
\end{array}
\right).
\end{equation}
The off-diagonal entries of $J$ are sign-stable and sign symmetric in $%
\mathbb{R}_+^3$.

Let
\begin{equation*}
\mathbb{K}=\left\lbrace (N,P_S, P_H ) \in \mathbb{R}^3: N \geq 0,P_S \geq 0,
P_H \leq 0 \right\rbrace.
\end{equation*}
The system is $\mathbb{K}$-competitive, since the matrix of the
time-reversed linearized system $-J$ is cooperative with respect to the cone
$\mathbb{K}$.
\end{proof}

\section{Convergence to the Rosenzweig-MacArthur model}

\label{Section6}

The time scale for the life expectancy (as well as the time scale needed for
the reproduction) is the year, while the time needed for the lynx to handle
the rabbit is measured by days (no more than one week). Therefore there is a huge
difference between the time scales for the vital dynamic and the
consumption dynamic.

The consumption of prey by the predator is a fast process compared to the
vital dynamic which is slow. In the model $\gamma^{-1}$ is the
average time spent by the predators to handle preys. $\gamma^{-1}$ should be
very small in comparison with the other parameters. Then it makes sense to make the following assumption.

\begin{Assumption}
\label{ASS6.1} Assume that
\begin{equation*}
\rho=\dfrac{\chi }{\varepsilon} \text{ and } \gamma=\dfrac{1}{\varepsilon}
\end{equation*}
with $\varepsilon \ll 1$ is small.
\end{Assumption}

Under the above assumption the system \eqref{1.1} becomes

\begin{equation}  \label{6.1}
\left\{
\begin{array}{ll}
\overset{\cdot}{N^\varepsilon}=\left( \beta _{N}-\mu _{N}\right)
N^\varepsilon -\delta (N^{\varepsilon})^{2} - \kappa N^\varepsilon
\,P_{S}^\varepsilon \vspace{0.1cm} \\
\overset{\cdot}{P^\varepsilon_{S}}=- (\mu _{P}+\eta ) P^\varepsilon_{S}-\dfrac{\chi }{
\varepsilon} \kappa N^\varepsilon P_{S}^\varepsilon+\dfrac{1}{\varepsilon} P_{H}^\varepsilon \vspace{0.1cm} \\
 
\overset{\cdot}{P^\varepsilon_{H}}=- \mu _{P}P^\varepsilon_{H} \hspace{0.6cm}+\dfrac{\chi }{
\varepsilon} \kappa N^\varepsilon P_{S}^\varepsilon - \dfrac{1}{\varepsilon} P_{H}^\varepsilon
+\beta _{P} \left(P_{S}^\varepsilon+P_{H}^\varepsilon \right)
\end{array}%
\right.
\end{equation}
and we fix the initial value
\begin{equation*}
N^\varepsilon(0)=N_0 \geq 0,\, P^\varepsilon_{S}(0)=P_{S0} \geq 0 \text{ and
} P^\varepsilon_{H}(0)=P_{H0} \geq 0.
\end{equation*}
The first equation of \eqref{6.1} is
\begin{equation}  \label{6.2}
\overset{\cdot}{N^\varepsilon}=\left( \beta _{N}-\mu _{N}\right)
N^\varepsilon -\delta (N^{\varepsilon})^{2}- \kappa N^\varepsilon
\,P_{S}^\varepsilon.
\end{equation}
Hence
\begin{equation}  \label{6.3}
\overset{\cdot}{N^\varepsilon}\leq \left( \beta _{N}-\mu _{N}\right)
N^\varepsilon.
\end{equation}
By summing the two last equations of \eqref{6.1} we obtain
\begin{equation}  \label{6.4}
\overset{\cdot}{P^\varepsilon}=\left(\beta_P-\mu_P \right)
P^\varepsilon-\eta P^\varepsilon_S
\end{equation}
and $P^\varepsilon_S \geq 0$ implies that
\begin{equation}  \label{6.5}
\overset{\cdot}{P^\varepsilon}\leq \left(\beta_P-\mu_P \right) P^\varepsilon.
\end{equation}
Therefore by using \eqref{6.3} and \eqref{6.5} we obtain the following
finite time estimation uniform in $\varepsilon$.

\begin{lemma}
\label{LE6.2} For each $\tau>0$ we can find a constant $M=M(\tau, N_0,P_0)>0$
(independent of $\varepsilon>0$) such that
\begin{equation}  \label{6.6}
0 \leq N^\varepsilon(t) \leq M \text{ and } 0 \leq P^\varepsilon(t) \leq M,
\forall t \in [0, \tau].
\end{equation}
and
\begin{equation}  \label{6.7}
\sup_{t \in \left[0,\tau\right]} \vert \overset{\cdot}{N^\varepsilon}(t)
\vert \leq M \text{ and } \sup_{t \in \left[0,\tau\right]} \vert \overset{%
\cdot}{P^\varepsilon}(t) \vert \leq M.
\end{equation}
\end{lemma}

\begin{proof}
We first deduce \eqref{6.6} by using the inequalities \eqref{6.3} and %
\eqref{6.5}. By using the fact $P_S \geq 0$ and $P_H \geq 0$ we have 
\begin{equation}  \label{6.8}
0 \leq P^\varepsilon_{S}(t) \leq M, \text{ and } 0 \leq P^\varepsilon_{H}(t)
\leq M, \forall t \in [0, \tau].
\end{equation}
Therefore by injecting these estimations into \eqref{6.2} and \eqref{6.4} we
deduce \eqref{6.7}.
\end{proof}

By using Lemma \ref{LE4.1}, and the Arzela-Ascoli theorem we deduce that we
can find a sequence $\varepsilon_n \to 0$ such that
\begin{equation*}
\lim_{n \to \infty }N^{\varepsilon_n}=N \text{ and } \lim_{n \to \infty
}P^{\varepsilon_n}=P
\end{equation*}
where the convergence is taking place in $C([0,\tau],\mathbb{R})$ for the
uniform convergence topology.

Moreover by using the fact that $P_{H}^\varepsilon=P^\varepsilon
-P_{S}^\varepsilon$, the $P_{S}^\varepsilon$-equation can be rewritten as
\begin{equation}  \label{6.9}
\overset{\cdot}{P^\varepsilon_{S}}=-\left( (\mu _{P}+\eta )+\dfrac{\chi }{%
\varepsilon} \kappa N^\varepsilon\right) P_{S}^\varepsilon+\dfrac{1}{%
\varepsilon} \left(P^\varepsilon -P_{S}^\varepsilon \right).
\end{equation}

By using \eqref{6.8}, the map $t \to P^\varepsilon_{S}(t)$ is bounded
uniformly in $\varepsilon$. So the family $\varepsilon_n \to
P^{\varepsilon_n}_{S} $ is bounded in $L^\infty\left(\left(0,\tau \right),
\mathbb{R}\right)$ which is the dual space of $L^1\left(\left(0,\tau
\right), \mathbb{R}\right)$. Therefore by using the
Banach-Alaoglu-Bourbaki's theorem, we can find a sub-sequence (denoted with
the same index) such that $\varepsilon_n \to P^{\varepsilon_n}_{S} $
convergences to $P_{S} \in L^\infty\left(\left(0,\tau \right), \mathbb{R}%
\right)$ for the weak star topology of $\sigma \left(
L^\infty\left(\left(0,\tau \right), \mathbb{R}\right), L^1\left(\left(0,\tau
\right), \mathbb{R}\right) \right)$. That is to say that for each $\chi \in
L^1\left(\left(0,\tau \right), \mathbb{R}\right)$
\begin{equation*}
\lim_{n \to \infty} \int_0^\tau \chi(t)
\left(P^{\varepsilon_n}_{S}(t)-P_{S}(t) \right)dt=0.
\end{equation*}
By multiplying \eqref{6.9} by $\chi \in C^1_c\left( \left(0,\tau \right),
\mathbb{R} \right)$ (the space $C^1$ functions with compact support in $%
\left(0,\tau \right)$) and by integrating over $[0,\tau]$ we obtain
\begin{equation*}
-\int_0^\tau \overset{\cdot}{ \chi}(t)P^{\varepsilon_n}_{S}(t)dt=\int_0^\tau
\chi(t) \left[ -\left( (\mu _{P}+\eta )+\dfrac{\chi }{\varepsilon_n} \kappa
N^{\varepsilon_n}(t)\right) P_{S}^{\varepsilon_n}(t)+\dfrac{1}{\varepsilon_n}
\left(P^{\varepsilon_n}(t) -P_{S}^{\varepsilon_n}(t) \right) \right] dt.
\end{equation*}
Hence by multiplying both sides by $\varepsilon_n$ and by taking the limit
when $n$ goes to infinity we obtain
\begin{equation*}
0=\int_0^\tau \chi(t) \left[ -\left( \chi \kappa N(t)\right) P_{S}(t)+
\left(P(t) -P_{S}(t) \right) \right] dt
\end{equation*}
and since $C^1_c\left( \left(0,\tau \right), \mathbb{R} \right)$ is dense in
$L^1\left( \left(0,\tau \right), \mathbb{R} \right)$ we deduce that
\begin{equation*}
P^{\varepsilon_n}_{S}(t) \overset{*}{ \rightharpoonup } \dfrac{1}{1+\chi
\kappa N(t)}P(t) \text{ as } n \to \infty.
\end{equation*}
By using the first equation of \eqref{6.1} and \eqref{6.4}, we have
\begin{equation*}
\begin{array}{l}
N^{\varepsilon_n}(t)=\dfrac{e^{\int_0^t \beta _{N}-\mu _{N} -\kappa
P_{S}^{\varepsilon_n}(\sigma)d\sigma} N_0}{1+\delta \int_0^t e^{\int_0^l
\beta _{N}-\mu _{N} -\kappa P_{S}^{\varepsilon_n}(\sigma)d\sigma} N_0 dl} ,
\\
P^{\varepsilon_n}(t)=e^{ \left(\beta_P-\mu_P \right)t} P_0-\int_0^t e^{
\left(\beta_P-\mu_P \right)\left( t-s\right)} \eta P^\varepsilon_S(\sigma)
d\sigma.%
\end{array}
\\
\end{equation*}
By taking the limit on both sides we deduce that
\begin{equation*}
\left\lbrace
\begin{array}{l}
\overset{\cdot}{N}=\left(\beta_{N}-\mu_{N} \right) N(t)-\delta N(t)^2-\dfrac{%
\kappa N(t)}{1+\chi \kappa N(t)} P(t), \\
\overset{\cdot}{P}=\left(\beta_P-\mu_P \right) P-\eta \dfrac{1}{ 1+\chi
\kappa N(t)} P.%
\end{array}
\right.
\end{equation*}
Therefore we obtain the following theorem.

\begin{theorem}
\label{TH6.3} For each fixed initial value $N_0 \geq 0$, $P_{S0} \geq 0$
and $P_{H0} \geq 0$. Let $\tau>0$ be fixed. Then the solution of \eqref{6.1}
satisfies the following
\begin{equation*}
\lim_{\varepsilon \to 0 }N^{\varepsilon}(t) =N(t) \text{ and }
\lim_{\varepsilon \to 0 }P^{\varepsilon}_S(t)+P^{\varepsilon}_H(t) =P(t)
\end{equation*}
where the limit is uniform on $[0, \tau]$ and $N(t)$ and $P(t)$ is the
solution of the Rosenzweig-MacArthur model
\begin{equation}  \label{6.10}
\left\lbrace
\begin{array}{l}
\overset{\cdot}{N}=\left(\beta_{N}-\mu_{N} \right) N(t)-\delta N(t)^2-\dfrac{%
\kappa N(t)}{1+\chi \kappa N(t)} P(t), \\
\overset{\cdot}{P}=\left(\beta_P-\mu_P-\eta \right) P+\eta \dfrac{\chi
\kappa N(t)}{ 1+\chi \kappa N(t)} P%
\end{array}
\right.
\end{equation}
with initial value
\begin{equation*}
N(0)=N_0 \text{ and } P(0)=P_{S0}+P_{H0}.
\end{equation*}
\end{theorem}

%


\begin{remark}
If instead of the model \eqref{1.1} we consider the following model
\begin{equation}  \label{6.11}
\left\{
\begin{array}{l}
\overset{\cdot}{N}=\left( \beta _{N}-\mu _{N}\right) N-\delta N^2 - \kappa
N^l \,P_{S} \\
\overset{\cdot}{P_{S}}=- (\mu _{P}+\eta ) P_{S}-\rho \kappa N^m P_{S}+\gamma
P_{H}, \\
\overset{\cdot}{P_{H}}= \beta_{P} \left( P_{S}+P_{H} \right) -\mu _{P} P_{H}
+\rho \kappa N^m P_{S}-\gamma P_{H}%
\end{array}
\right.
\end{equation}
Then by using the same procedure above we obtain a convergence result to the most classical predator prey model
\begin{equation}  \label{6.12}
\left\lbrace
\begin{array}{l}
\overset{\cdot}{N}=\left(\beta_{N}-\mu_{N} \right) N(t)-\delta N(t)^2-\dfrac{%
\kappa N(t)^l}{1+\chi \kappa N(t)^m} P(t), \\
\overset{\cdot}{P}=\left(\beta_P-\mu_P-\eta \right) P+\eta \dfrac{\chi
\kappa N(t)^m}{ 1+\chi \kappa N(t)^m} P.%
\end{array}
\right.
\end{equation}
By choosing $l=m$ we obtain the classical Holling's type functional
response. 
\end{remark}

\section{Application to the snowshoe hares and lynxes} 
\label{Section7}
In this section we reconsider predator-prey system form by the hares (prey) and lynxes (predator) in the years 1900-1920 recorded by the Hudson Bay Company. The data are available for example in \cite{Deuflhard-Roblitz}.

\begin{table}[H] \centering
		\begin{tabular}{ccc}
			\doubleRule
			\textbf{Year} & \textbf{Hares} (in thousands)  &  \textbf{Lynx} (in thousands) \\
			\hline
			$1900$ & $ 30$  & $4$ 
			\\
			$1901$ & 	$47.2$ & 	$6.1$
			\\
			$1902$ & 	$70.2$ & 	$9.8$
			\\
			$1903$	& $77.4$ &	$35.2$
			\\
			$1904$	& $36.3$ &	$59.4$\\
			$1905$	& $20.6$	 & $41.7$\\
$1906$	& $18.1$ & 	$19$\\
$1907$	& $21.4$ & 	$13$\\
$1908$ &	 $22$	&  $8.3$\\
$1909$	& $25.4$	& $9.1$\\
$1910$	& $27.1$ &	$7.4$\\
$1911$ &	 $40.3$ & 	$8$\\
$1912$ &	 $57$	& $12.3$\\
$1913$ &	 $76.6$	 & $19.5$\\
$1914$ &	 $52.3$	 & $45.7$\\
$1915$ &	 $19.5$ &	$51.1$\\
$1916$ & 	$11.2$	& $29.7$\\
$1917$ &	 $7.6$	& $15.8$\\
$1918$ & 	$14.6$	& $9.7$\\
$1919$ &	 $16.2$ & 	$10.1$\\
$1920$ & 	$24.7$ & 	$8.6$\\	
			\end{tabular} 
		\caption{\textit{Numbers of hares (prey) and lynxes (predator) in the years 1900-1920
recorded by the Hudson Bay Company}}\label{Table1}
	\end{table}
	
The limit model obtain for $\varepsilon$ small enough is given by  	
\begin{equation}  \label{7.1}
\left\lbrace
\begin{array}{l}
\overset{\cdot}{N}=\left(\beta_{N}-\mu_{N} \right) N \left( 1-\dfrac{N}{\delta} \right) 
-\dfrac{\kappa P N}{1+\chi \kappa N} , \\
\overset{\cdot}{P}=\left(\beta_P-\mu_P-\eta \right) P+\eta \dfrac{\chi
\kappa P N}{ 1+\chi \kappa N}
\end{array}
\right.
\end{equation}
with initial value
\begin{equation*}
N(0)=N_0=30 \times 10^3 \text{ and } P(0)=P_0=4 \times 10^3.
\end{equation*} 

\begin{table}[H] \centering
		\begin{tabular}{cccccc}
			\doubleRule
			\textbf{Symbol}  &  \textbf{Interpretation} & \textbf{Value}  &\textbf{Unit} &  \textbf{Method}  \\
			\hline
			$ 1/\mu_N $  
			& Life expectancy of hares
			& $1$
			&  year 
			& fixed
			\\
			$ \beta_N $  
			& Birth rate of hares
			& $1.6567$
			&  number of new born/year
			& fitted
			\\
			$ \delta $  
			& Carrying capacity of hares
			& $303000$
			&  year 
			& fitted
			\\
			$ \kappa $  
			& 
			& $3.2\,\times 10^{-5}$
			& 
			& fitted
			\\
			$ \chi $  
			& 
			& $0.11$
			&  
			& fitted
			\\
			$ 1/\mu_P $  
			& Life expectancy of Lynx
			& $7$
			&  year 
			& fixed
			\\
			$ \beta_P $  
			& Birth rate of Lynx
			& $8.5127$
			&  number of new born/year
			& fitted
			\\
			$ \eta$  
			& Extra mortality of searching Lynx
			& $9.24$
			&  year$^{-1}$
			& fitted
			\\
			$ \beta_P-\mu_P-\eta$  
			& Growth of searching lynx
			& $-0.8702$
			& 
			& fitted
			\\
			$ \eta \chi$  
			& Convertion rate 
			& $1.0164$
			& 
			& fitted \\
			\doubleRule
		\end{tabular} 
		\caption{\textit{List parameters for the model \eqref{7.1}, their interpretations, values and symbols. In this table we have fixed $\mu_N$ and $\mu_P$ and we have obtain all the remaining parameters by using a least square method between the data in Table \ref{Table1} the solution of the model \eqref{7.1}. The life expectancy of Snowshoe Hares is not known \cite{Chitty, Gilpin}. Here we fix the life expectancy of hares to be $1$ year (similarly to   \cite{Yukon}). In the wild a Canadian Lynx can live up to $14$ years. Here we fix the life expectancy to be $7$ years (see \cite{Elton} for more result). A Canadian lynx can have between 1 and 8 new babies \cite{Jackson}. So the estimation obtained for the birth rate of lynxes is still reasonable. }}\label{Table2}
	\end{table}

\begin{figure}[H]
\centering
\includegraphics[width=6in]{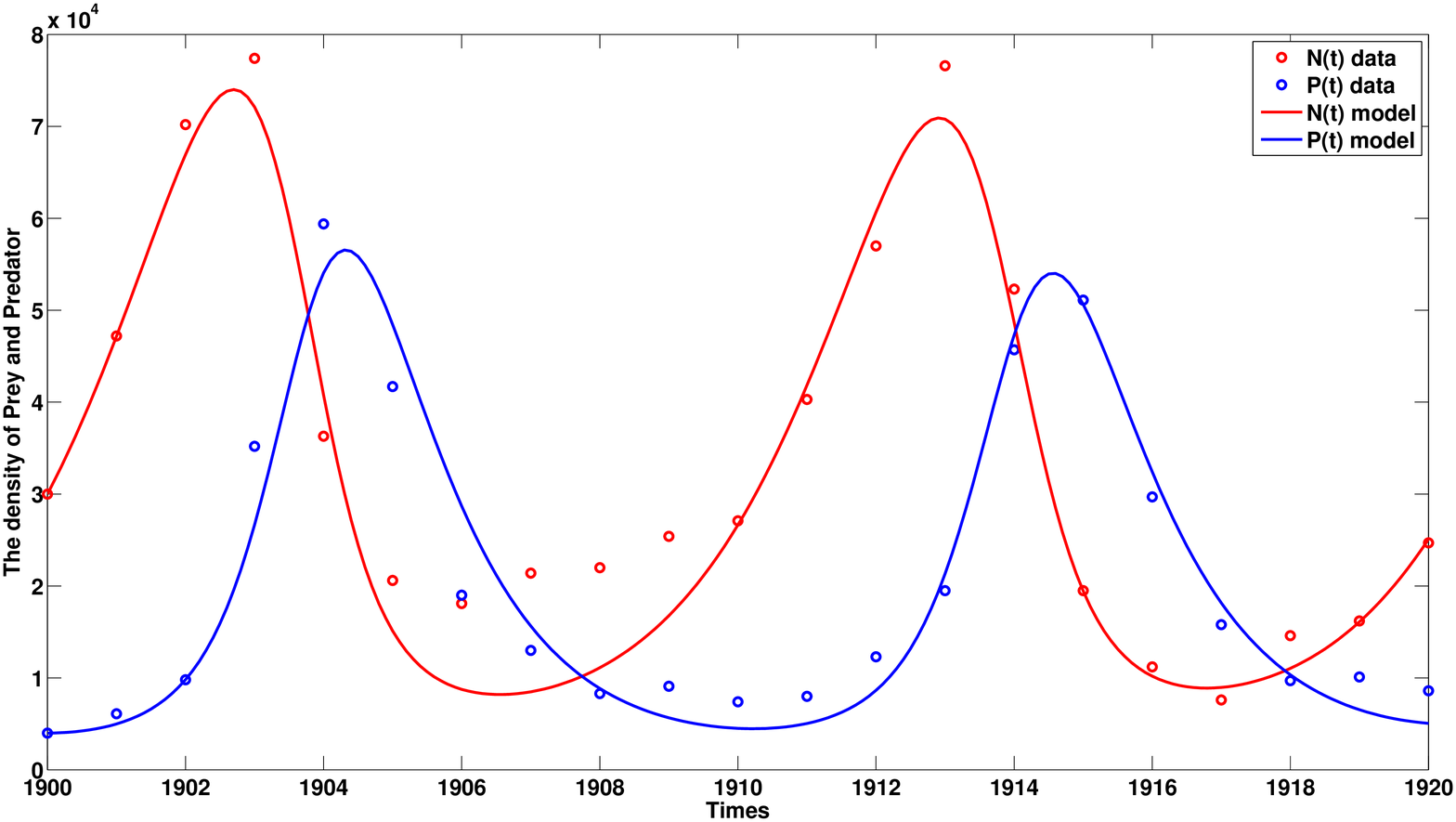} 
\caption{\textit{In this figure we run a simulation of the model \eqref{7.1} (solide lines)  compared with the data (circles). }}
\label{fig1}
\end{figure}

In section \ref{Section6}, we proved that the model \eqref{7.1} can be obtained as singular limit (when $\varepsilon \to 0$) of the following model 

\begin{equation}  \label{7.2}
\left\{
\begin{array}{ll}
\overset{\cdot}{N^\varepsilon}=\left( \beta _{N}-\mu _{N}\right) N^\varepsilon \left(
1 - \dfrac{N^{\varepsilon}}{\delta} \right) - \kappa N^\varepsilon
\,P_{S}^\varepsilon \vspace{0.1cm} \\
\overset{\cdot}{P^\varepsilon_{S}}=- (\mu _{P}+\eta ) P^\varepsilon_{S}-\dfrac{\chi }{
\varepsilon} \kappa N^\varepsilon P_{S}^\varepsilon+\dfrac{1}{\varepsilon} P_{H}^\varepsilon \vspace{0.1cm} \\
 
\overset{\cdot}{P^\varepsilon_{H}}=- \mu _{P}P^\varepsilon_{H} \hspace{0.6cm}+\dfrac{\chi }{
\varepsilon} \kappa N^\varepsilon P_{S}^\varepsilon - \dfrac{1}{\varepsilon} P_{H}^\varepsilon
+\beta _{P} \left(P_{S}^\varepsilon+P_{H}^\varepsilon \right)
\end{array}%
\right.
\end{equation}
and we fix the initial value
\begin{equation*}
N^\varepsilon(0)=N_0=30 \times 10^3  \geq 0,\, P^\varepsilon_{S}(0)=P_{S0} \geq 0 \text{ and
} P^\varepsilon_{H}(0)=P_{H0} \geq 0.
\end{equation*}
In Theorem \ref{TH6.3} we proved that for $\varepsilon$ small enough 
\begin{equation}  \label{7.3}
P^{\varepsilon}_{S}(t) \simeq \dfrac{1}{1+\chi
\kappa N(t)}P(t) \text{ and } P^{\varepsilon}_{R}(t) \simeq \left(1- \dfrac{1}{1+\chi
\kappa N(t)} \right) P(t)=\dfrac{\chi \kappa N(t)}{1+\chi
\kappa N(t)}  P(t).
\end{equation}
By using the value for $\chi \kappa$ estimated in Table \ref{Table2},  we obtain the following initial values for the model \eqref{7.2} 
\begin{equation} \label{7.4}
P^{\varepsilon}_{S0} =\dfrac{P_0}{1+\chi
\kappa N_0}=\dfrac{4 \times 10^3}{1+1.0164 \times 30 \times 10^3 } \text{ and } P^{\varepsilon}_{R0}= \dfrac{\chi \kappa N_0}{1+\chi
\kappa N_0}  P_0=\dfrac{1.0164 \times 30 \times 10^3}{1+1.0164 \times 30 \times 10^3} 4 \times 10^3.
\end{equation}

\begin{figure}[H]
\centering
\includegraphics[width=6in]{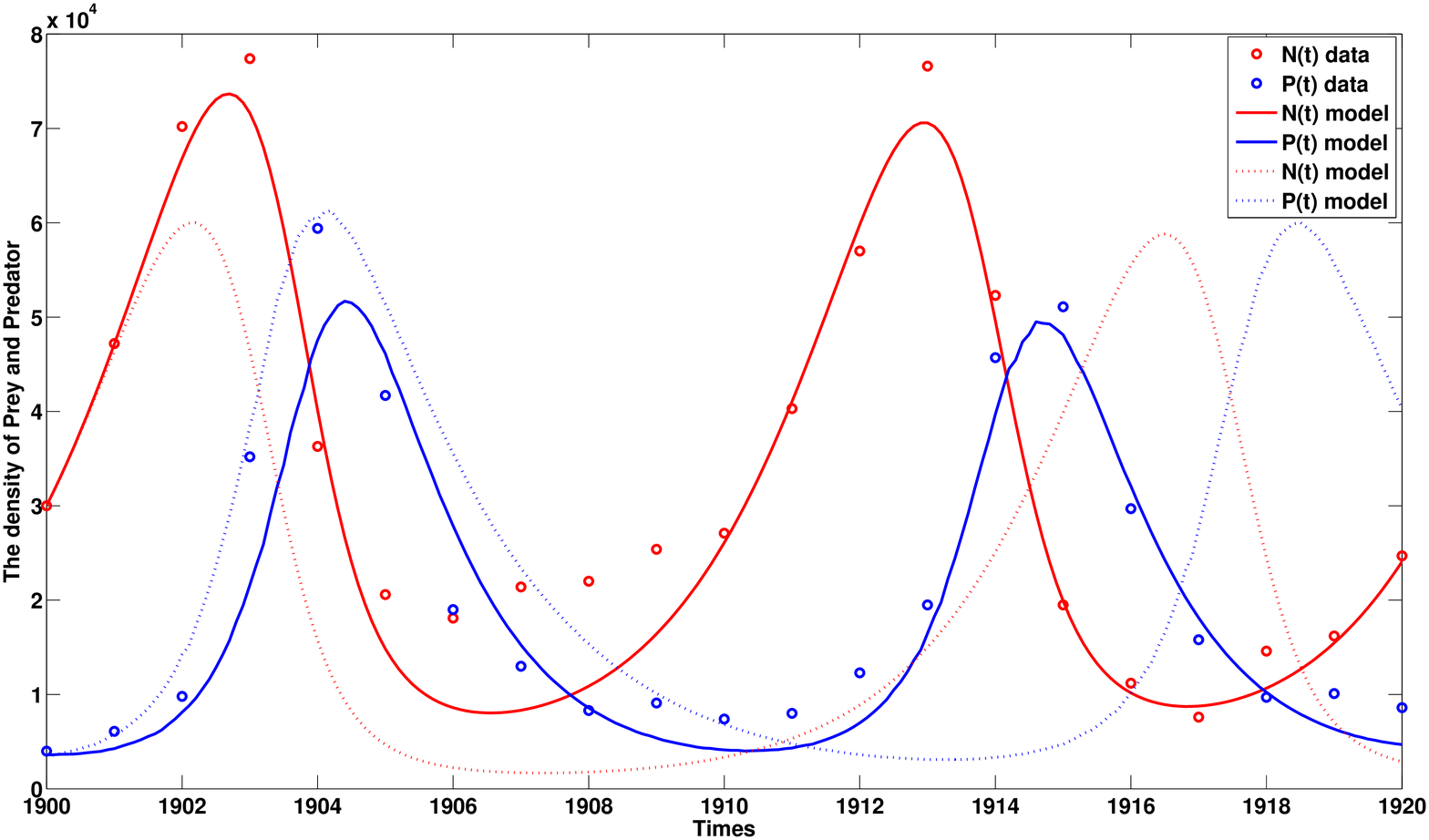} 
\caption{\textit{In this figure we run a simulation of the model \eqref{7.2} (solide and dotted lines)  compared with the data (circles). The solide lines correspond to $\varepsilon=10^{-4}$ and the dotted lines correspond to $\varepsilon=5.10^{-3}$.}}
\label{fig2}
\end{figure}

\begin{figure}[H]
\centering
\includegraphics[width=6in]{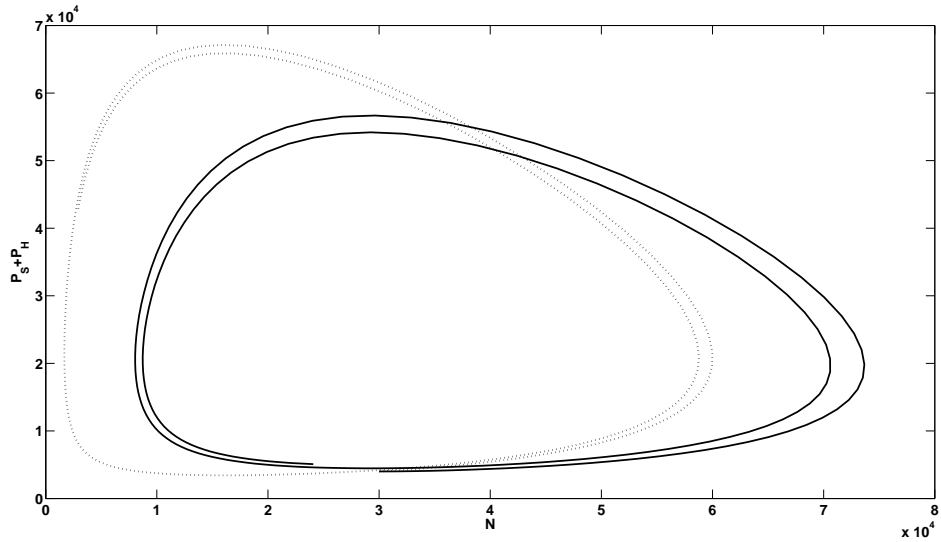} 
\caption{\textit{In this figure we run a simulation of the model \eqref{7.2} with $\varepsilon=10^{-4}$ for solide line and with $\varepsilon=5.10^{-3}$ for dotted line. }}
\label{fig3}
\end{figure}
From Figures \ref{fig1} and \ref{fig2}, we can see that $\varepsilon$ does not need to be very small ($\varepsilon= 10^{-4}$) to get an almost
perfect match of our model \eqref{7.2} with the Rosenzweig-MacArthur model \eqref{7.1}. Our simulations for hares and lynxes fit the data reported by the Hudson Bay Company. As we mentioned the main advantage with the model \eqref{7.2} is that we can separate the vital dynamic and consumption of preys (hares) to describe the behavior of the predators (lynxes). From our model \eqref{7.2}, people can study the interaction between predator and prey in detail and get more information.

\end{document}